\documentclass[reqno]{amsart}
\UseRawInputEncoding
\makeatletter
\@namedef{subjclassname@2020}{\textup{2020} Mathematics Subject Classification}
\makeatother

\usepackage{amssymb}
\usepackage[all]{xy}
\usepackage{enumitem}
\usepackage{hyperref}
\usepackage{graphicx,xcolor}
\usepackage{tikz}
\usepackage{tikz-cd}

\def\commentout#1{{}}

\def\Z{{\mathbb Z}}

\def\N{{\mathbb N}}

\def\cC{{\mathcal C}}
\def\cD{{\mathcal D}}
\def\cE{{\mathcal E}}

\newcommand\Hom{\operatorname{Hom}}

\newcommand\FinSets{{\operatorname{FinSets}}}

\newcommand\id{{\operatorname{id}}}

\newcommand\Dec{{{\operatorname{Dec}}}}

\newtheorem{theorem}{Theorem}[section]
\newtheorem{lemma}[theorem]{Lemma}
\newtheorem{proposition}[theorem]{Proposition}
\newtheorem{corollary}[theorem]{Corollary}

\theoremstyle{definition}
\newtheorem{definition}[theorem]{Definition}

\theoremstyle{remark}

\numberwithin{equation}{section}

\usepackage{amsmath} 
\begin{document}
\title[Locally-finite extensive categories and their semi-rings]
{Locally-finite extensive categories, their semi-rings, 
and decomposition to connected objects}

\author{Shoma Fujino}
\address{Mathematics Program \\ 
Graduate School of Advanced Science and Engineering\\
Hiroshima University, 739-8526 Japan}
\email{shomafujino0729@gmail.com}

\author{Makoto Matsumoto}
\address{Mathematics Program \\ 
Graduate School of Advanced Science and Engineering\\
Hiroshima University, 739-8526 Japan}
\email{m-mat@math.sci.hiroshima-u.ac.jp}

\thanks{
The second author is partially supported by JSPS
Grants-in-Aid for Scientific Research
JP26310211 and JP18K03213.
}

\date{\today}

\begin{abstract}
Let $\cC$ be the category of finite graphs.
Lov\`{a}sz shows that the semi-ring of isomorphism classes
of $\cC$ (with coproduct as sum, and product as multiplication)
is embedded into the direct product of the semi-ring of natural
numbers. Our aim is to generalize this result to other categories.
For this, one crucial property is that every object decomposes to a finite coproduct
of connected objects. We show that a locally-finite extensive category
satisfies this condition. Conversely, a category where
any object is decomposed into a finite coproduct
of connected objects is shown to be extensive.
The decomposition turns out to be unique.
Using these results, we give some sufficient conditions that
the semi-ring (the ring) of isomorphism classes of a locally finite
category embeds to the direct product of natural numbers (integers, respectively).
Such a construction of rings from a category is
a most primitive form of Burnside rings and Grothendieck rings.
\end{abstract}

\maketitle
\setcounter{tocdepth}{3}
\tableofcontents
\section{Introduction and main results}
Let $\N$ denote the semi-ring of non-negative integers.
Here, we mean by semi-ring a set $S$ equipped with 
two commutative binary operators $+$, $\times$,
where both operators give monoid structures on $S$, with 
unit denoted by $0$, $1$, respectively, such that the operators
satisfy the distributive laws and $0\times x=0$ for every $x\in S$.
The same object is often called a rig, because 
it is similar to the ring without the inverse with respect
to $+$ (thus, lacking ``negative,'' hence called ``ring $-$ negative = rig'').
In this paper, a monoid means a commutative monoid.

In Lov\`{a}sz's seminal work \cite[(3.6)Theorem]{LOVASZ-OPERATION},
he gives an injective semi-ring homomorphism of the ``Burnside semi-ring''
of the category\footnote{There seems no decisive name of 
this classical construction of semi-ring, e.g., appearing in the 
construction of Burnside rings. The definition is given 
shortly after as $\Dec(\cC,+,\times)$.}
of finite directed graphs (and more generally, of finite relational structures)
to an infinite product of the copies of the semi-ring $\N$.

The motivation of this study is to generalize this result 
in terms of category theory. We shall show that the notion 
of extensive categories \cite{CARBONI} plays a central role.
We shall define terminologies and state our main results, while
observing the Lov\`{a}sz's work. 
\begin{definition}
A category $\cC$ is locally finite, 
if for any two objects, the corresponding hom-set
is a finite set.
\end{definition}
For a category $\cC$, we denote by $\Dec(\cC)$
the set of isomorphism classes of $\cC$. ($\Dec$ means the
de-categorification.)
Let $\cC$ be a category with finite coproduct. This requires
the existence of an initial object $0$. We denote by $A+B$
the coproduct.
Then, $\Dec(\cC)$ has a monoid structure inherited
from $+$, denoted by $(\Dec(\cC),+)$.
This kind of constructions is well-known,
in the context of Burnside rings and Grothendieck groups.
Here we deal with most primitive cases, arising from 
the coproducts and products in a category.
\begin{definition}
Let $\cC$ be a locally finite category, and $D$ its object. 
We define
$$
h_D: \Dec(\cC) \to \N, \quad [A] \mapsto \#\Hom(D,A),
$$
where $\#$ denotes the cardinality of the set.
\end{definition}

The following gives the definition of connectedness.
\footnote{This definition requires less than the standard
that requires preservation of infinite coproducts \cite[\S5.2, P.453]{MONOIDALT}, 
but appropriate in
the present context where infinite coproducts 
may not exist. Both become equivalent if the category is
infinitary extensive, which is called ``extensive''
in \cite[5.1.1 Definition, P.449]{MONOIDALT}, see 
the footnote on Proposition~\ref{prop:connectivity} for the equivalence.}
\begin{definition} 
\label{def:conn}
An object $C\in \cC$ is said to be connected,
if the natural mapping
\begin{equation}\label{eq:conn}
\Hom(C,A)\coprod \Hom(C,B) \to \Hom(C, A+B) 
\end{equation}
is bijective. 
\end{definition}
By definition, the following holds.
\begin{proposition}
Let $\cC$ be a locally finite category with finite
coproducts. Let $C$ be a connected object. 
Then, 
$$
h_C:(\Dec(\cC),+) \to (\N,+)
$$ 
is a monoid morphism.
\end{proposition}
Suppose that $\cC$ has finite direct products (and hence a terminal object). It induces
a monoid structure $(\Dec(\cC),\times)$.
By definition, 
$$
\Hom(D, A\times B) \to \Hom(D,A)\times \Hom(D,B)
$$
is a bijection for any object $D$, and we have the following.
\begin{proposition}
Let $\cC$ be a locally finite category with finite
products. Let $D$ be any object. 
Then,
$$
h_D:(\Dec(\cC),\times) \to (\N,\times)
$$ 
is a morphism of monoids.
\end{proposition}

The following definition is according to Carboni et.\ al., 
see \cite[Definition~3.1 and Proposition~3.2]{CARBONI}.
\begin{definition}
Let $\cC$ be a category with finite coproducts and finite
products. If the natural morphism
$$
A\times B + A\times C \to A\times (B+C)
$$
is an isomorphism, then $\cC$ is said to be distributive.
In a distributive category, it holds that
$$
A \times 0 \cong 0.
$$
\end{definition}
As a consequence, we have
\begin{proposition}
Let $\cC$ be a distributive category.
Then, $(\Dec(\cC),+,\times)$ is a semi-ring. 
If $\cC$ is moreover locally finite and $C$ is a connected object, then
$$
h_C:(\Dec(\cC),+,\times) \to (\N,+,\times)
$$
is a semi-ring homomorphism.
\end{proposition}
For example, the category $\FinSets$ of finite sets is distributive, 
and the semi-ring $(\Dec(\FinSets),+,\times)$ is naturally isomorphic to $(\N,+,\times)$.
This might be considered as a definition of $\N$.
In this case, if we take $C$ as a singleton, $h_C$ gives an isomorphism.

In a category, it holds that
$$
A \cong B \Rightarrow \#\Hom(X,A)=\#\Hom(X,B) \mbox{ for all objects } X. 
$$
A category where the converse holds is said to be combinatorial. 
\begin{definition} (\cite[1.7~Definition]{PULTR})
\label{def:combinatorial}
A locally finite category $\cC$ is said to be
combinatorial, if for all objects $X$
$$
\#\Hom(X,A)=\#\Hom(X,B)
$$
hold then $A$ is isomorphic to $B$.
\end{definition}
Lov\`{a}sz \cite{LOVASZ-OPERATION} 
proved that the categories of operations with finite structures
(including the category of finite graphs) are combinatorial.
Various sufficient conditions for combinatoriality are known.
See Pultr \cite{PULTR}, Isbell \cite{ISBELL}, 
Dawar, Jakl, and Reggio \cite{DAWAR},
Reggio \cite{REGGIO}, and Fujino and Matsumoto \cite{FUJINO}.

If any object $X$ of $\cC$ is a finite coproduct of
connected objects $C_1,\ldots,C_n$, then 
$$
\Hom(X,-)\cong \prod_{i=1}^n\Hom(C_i,-)
$$
follows, and to show $A\cong B$, it suffices to show $h_C(A)=h_C(B)$
for every connected $C$ if $\cC$ is combinatorial. Then
$h_C$ is a monoid morphism, and hence $\Dec(\cC,+)$ embeds
into a direct product of copies of $(\N,+)$.
Thus we have the following immediate consequence,
which is a straightforward generalization of Lov\`{a}sz's arguments.
\begin{proposition}\label{prop:lovasz}
Let $\cC$ be a locally finite category with finite coproducts.
Suppose that $\cC$ is combinatorial, and any object is a 
coproduct of a finite number of connected objects. 
Choose a representative system $C_i$ $(i\in I)$ from 
the set of isomorphism classes of connected objects. Then,
$$
\prod_{i\in I}h_{C_i}:(\Dec(\cC),+) \to \prod_{i\in I}(\N,+)
$$
is an injective homomorphism of monoids.
If, moreover, $\cC$ is distributive, then
$$
\prod_{i\in I}h_{C_i}:(\Dec(\cC),+,\times) \to \prod_{i\in I}(\N,+,\times)
$$
is an injective homomorphism of semi-rings.
\end{proposition}
Thus, it becomes important whether each object
is decomposed as a coproduct of a finite number of 
connected objects. 
The following are main results of this paper. 
\begin{theorem}\label{th:main} $ $

\begin{enumerate}
\item \label{enum:main2}
Let $\cC$ be a locally-finite extensive category.
Then, every object is a coproduct of a finite number of connected objects.
\item \label{enum:main1}
Let $\cC$ be a category with finite coproducts such that every object $D$ is
a finite coproduct of connected objects. Then, $\cC$
is an extensive category. 
\item \label{enum:main3}
Let $\cC$ be an extensive category. 
If an object $D$ decomposes to
a finite coproduct of connected objects,
then the decomposition is unique, up to ordering and isomorphisms of each component.
\end{enumerate}
\end{theorem}
The notion of extensive categories is established by
Carboni et.\ al.\ \cite{CARBONI}, and widely accepted
as a natural generalization of distributive categories.
For example, any topos is extensive.
We shall give a definition of extensive categories in the next section.
The following proposition is known.
\begin{proposition}\label{prop:ext-dist} (\cite[Proposition~4.5]{CARBONI})
An extensive category with finite products is distributive.
\end{proposition}
Thus, our main theorem yields the following theorem.
\begin{theorem}\label{th:main-emb}
Let $\cC$ be a locally finite, extensive, and combinatorial category.
Let $C_i$ $(i\in I)$ be representatives of the set of 
isomorphism classes of connected objects.
Then,
$$
\prod_{i\in I}h_{C_i}:(\Dec(\cC),+) \to \prod_{i\in I}(\N,+)
$$
is an injective morphism of monoids.
If, moreover, $\cC$ has finite products, then
$$
\prod_{i\in I}h_{C_i}:(\Dec(\cC),+,\times) \to \prod_{i\in I}(\N,+,\times)
$$
is an injective morphism of semi-rings.
\end{theorem}
There is a universal way to obtain a group from a monoid
(the left adjoint to the forgetful functor),
called the Grothendieck group of the monoid. This makes a monoid into a group, 
here denoted by
$$
(\Dec(\cC),+) \mapsto (\Dec(\cC),+,-).
$$
The monoid $\N$ is transferred to the additive group $\Z$.
The same construction makes a semi-ring into a ring,
which we denote
$$
(\Dec(\cC),+,\times) \mapsto (\Dec(\cC),+,-,\times).
$$

Theorem~\ref{th:main}(\ref{enum:main3}) implies the following
\begin{corollary}\label{cor:main}
Under the conditions of Theorem~\ref{th:main},
we have a canonical monoid injection 
\begin{equation}\label{eq:grothen-inj}
(\Dec(\cC),+) \hookrightarrow (\Dec(\cC),+,-), 
\end{equation}
and if moreover $\cC$ is combinatorial,
then we have an injective group homomorphism
\begin{equation}\label{eq:hom-count-inj}
\prod_{i\in I}h_{C_i}:(\Dec(\cC),+,-) \to \prod_{i\in I}(\Z,+). 
\end{equation}
(Note that $\prod_{i\in I}h_{C_i}$ here is the extension to
the Grothendieck group.)
In this case, if $\cC$ has finite products, then (\ref{eq:grothen-inj}) is
an injection of semi-rings
$$
(\Dec(\cC),+,\times) \hookrightarrow (\Dec(\cC),+,-,\times),
$$
and (\ref{eq:hom-count-inj}) is an injective ring homomorphism
$$
(\Dec(\cC),+,-,\times) \to \prod_{i\in I}(\Z,+,\times).
$$
\end{corollary}
For example, 
consider the functor categories $\FinSets^\cD$
for a finite category $\cD$. 
It is locally finite, and is shown to be
combinatorial, by methods in for example \cite{FUJINO} \cite{PULTR}. 
It is an elementary topos, and hence is extensive \cite{CARBONI},
and has finite products.
By Theorems~\ref{th:main}, \ref{th:main-emb} and Corollary~\ref{cor:main},
$\Dec(\FinSets^\cD,+,-,\times)$ embeds into a direct product of 
copies of $\Z$. In particular, there are no nilpotent elements.
By the same argument, a similar embedding is obtained for the category of 
$G$-finite sets where $G$ is a group, which is known as a theorem 
by Burnside.

\section{Proof of main results}
\subsection{Extensive categories}
This section follows Carboni et.\ al.\ \cite{CARBONI}.
For a coproduct diagram
$$
A \to A + B \leftarrow B,
$$
we call $A \to A+B$ the coprojection, and denote by $i_A$.
\begin{definition}\label{def:extensive}
A category $\cE$ is an extensive category, if it has finite coproducts
and satisfies the following conditions.
\begin{enumerate}
 \item For any morphism $f:A \to X_1+X_2$, there is a pullback
$X_1\prod_{X_1+X_2} A$ along the coprojection.
 \item Suppose that the following diagram commutes.
       \begin{equation}\label{eq:extensive}
	\begin{tikzcd}
	 A_1\ar{r}\ar{d} & A\ar{d}{f}& A_2\ar{l}\ar{d}\\
	 X_1\ar{r}[swap]{i_{X_1}} & X_1+X_2 & X_2\ar{l}{i_{X_2}}.
	\end{tikzcd}
       \end{equation}
Then, the top row is a coproduct, if and only if the both squares are pullbacks.
\end{enumerate}
\end{definition}

\begin{proposition}\label{prop:sum}\cite[Proposition~2.6]{CARBONI}

In an extensive category, the following three squares are pullbacks.
In particular, the coprojections are mono:
 		\begin{center}
			\begin{tikzcd}
				A\ar{r}\ar{d} & A\ar{d} & B\ar{r}\ar{d} & B\ar{d} &
				0\ar{r}\ar{d} & B\ar{d} \\
				A\ar{r} & A+B & B\ar{r} & A+B & A\ar{r} & A+B.
			\end{tikzcd}
		\end{center}
\end{proposition}
\begin{proof}
This follows from the following commutative diagram 
with rows being coproducts, and the second condition 
of Definition~\ref{def:extensive}:
 \begin{center}
  \begin{tikzcd}
   0\ar{r}\ar{d} & A_2\ar{d} & A_2\ar{l}\ar{d}\\
   A_1\ar{r} & A_1+A_2 & A_2\ar{l}.
  \end{tikzcd}
 \end{center}
\end{proof}

\begin{proposition}\label{prop:strict}\cite[Proposition~2.8]{CARBONI}

In an extensive category, any morphism $A \to 0$ is an isomorphism. 
\end{proposition}
\begin{proof}
Take $\alpha: A\to 0$. In the commutative diagram
 \begin{center}
  \begin{tikzcd}
   A\ar{r}{\id_A}\ar{d}[swap]{\alpha} & A\ar{d}{\alpha}
   & A\ar{l}[swap]{\id_A}\ar{d}{\alpha}\\
   0\ar{r}[swap]{\id_0} & 0 & 0\ar{l}{\id_0},
  \end{tikzcd}
 \end{center}
the bottom row is a direct product, the two squares are pullbacks,
and hence the second condition of Definition~\ref{def:extensive}
shows that the top row is a coproduct. Then, we have $A \stackrel{\id_A}{\to} A$
and $A \stackrel{\alpha}\to 0 \to A$, and by the universality of the top row
as a coproduct implies that these two coincide. Thus 
$A \stackrel{\alpha}\to 0 \to A$ is the identity, and $0 \to A \stackrel{\alpha}{\to} 0$
is the identity, hence $A\cong 0$.
\end{proof}

\begin{lemma}\label{lem:hom-sum}
In an extensive category, if $f:X \to A$ and $g:X \to B$
satisfy $i_A\circ f=i_B\circ g: X \to A+B$, then
$X\cong 0$. 
\end{lemma}
\begin{proof}
Proposition~\ref{prop:sum} implies that $f$ and $g$ factor through the
pullback of $A \to A+B \leftarrow B$, which is $0$, hence
the claim follows from Proposition~\ref{prop:strict}. 
\end{proof}

The following proposition is a slight variant of (\cite[5.2.1 Theorem, P.453]{MONOIDALT}).
\footnote{
To be precise, \cite{MONOIDALT} deals with
infinitary extensive categories. This proposition proves that
the notion of connectedness here coincides with that in \cite[\S5.2, P.453]{MONOIDALT}
for infinitary extensive categories, since the contraposition of this
proposition gives the equivalent condition (vi) to the connectedness 
in \cite[5.2.1 Theorem, P.453]{MONOIDALT}.} 
\begin{proposition} 
\label{prop:connectivity}
Let $X$ be an object of an extensive category $\cE$. 
The following are equivalent.
\begin{enumerate}
 \item $X$ is not connected. \label{enum:former}
 \item $X \cong 0$, or, there are two objects 
$U\not\cong 0, \ V\not\cong 0$ such that $U+V \cong X$. \label{enum:latter}
\end{enumerate}
\end{proposition}
\begin{proof}
Suppose (\ref{enum:latter}). If $X\cong 0$, then 
$$
\Hom(0,A)\coprod \Hom(0,B) \to \Hom(0, A+B)
$$ 
is not injective, hence by Definition~\ref{def:conn}, $X$ is not
connected. If $X\cong U+V$, in 
$$
\Hom(U+V,U)\coprod \Hom(U+V,V) \to \Hom(U+V, U+V),
$$ 
we shall show that 
$\id_{U+V}$ in the right-hand side does not come from the left-hand side,
hence $U+V$ is not connected.
If it does, we may assume that it comes from $\Hom(U+V,V)$, i.e.,
$U+V \to V \stackrel{i_U}{\to} U+V$ is an identity.
Then $i_U$ is a split epi, and mono by Proposition~\ref{prop:sum},
hence is an isomorphism. This implies that in the right most diagram 
of Proposition~\ref{prop:sum}, the bottom arrow is an isomorphism, 
and so is the top, which implies $V\cong 0$, contradicting (\ref{enum:latter}).
Thus, (\ref{enum:latter}) implies (\ref{enum:former}).

Suppose (\ref{enum:former}). (\ref{eq:conn}) for $C=X$ in Definition~\ref{def:conn}
is not bijective. Suppose that it is not injective.
Since coprojections are mono, this implies that
there are $X \to A$ and $X \to B$, which give one same morphism
$X \to A+B$ after composing with coprojections. 
By Lemma~\ref{lem:hom-sum}, $X\cong 0$, which proves the claim
in this case. Suppose that 
(\ref{def:conn}) is not surjective. Take an $f:X \to A+B$
which does not come from the left. We take pullbacks
       \begin{center}
		\begin{tikzcd}
			X_A \ar{r}\ar{d} & X\ar{d}{f}& X_B\ar{l}\ar{d}\\
			A  \ar{r}[swap]{i_{A}} & A+B & B\ar{l}{i_{B}}.
		\end{tikzcd}
       \end{center}
By definition
of an extensive category, we have $X\cong X_A + X_B$. Suppose 
that $X_A\cong 0$. Then, $X\cong 0+X_B\cong X_B$ implies that
$f$ is in the image of the composition $X \to X_B \to B$ in $\Hom(X,B)$,
which contradicts the assumption that $f$ is not in the image.
Thus, $X_A\not\cong 0$. Similarly, $X_B\not\cong 0$, and hence
(\ref{enum:latter}) holds.
\end{proof}

\subsection{Decomposition to connected objects}
This section proves (\ref{enum:main2}) and (\ref{enum:main3}) in Theorem~\ref{th:main}.
\begin{proposition}\label{finite_sum_decomposition}
Suppose that an extensive category $\cE$
is locally finite. Then, every object $X$ is
a coproduct of a finite number of connected objects.
\end{proposition}
\begin{proof}
Suppose that 
$X\cong X_1+X_2+\cdots+X_n$ with $X_k\not\cong 0$.
Let $i_1, i_2$ be the coprojections from $X$ to $X+X$ (to the left component,
to the right component, respectively). 
By Lemma~\ref{lem:hom-sum},
$i_1i_{X_k}\neq i_2i_{X_k}$ follows, and hence
\begin{align}
  \#\Hom(X_k,X+X)\geq 2.
\end{align}
Thus		
\begin{align}
 \#\Hom(X,X+X)&=\prod^n_{i=1}\#\Hom(X_i,X+X)\geq 2^n.
\end{align}
By locally finiteness, $2^n$ is bounded above, and so is $n$.
If $X$ is connected, then it is a coproduct of one connected object.
If $X$ is not connected, then 
Proposition~\ref{prop:connectivity}
implies that $X\cong 0$ or $X \cong U,V$, $U\neq 0, V\neq 0$.
In the former case, $X$ is a coproduct of zero of connected objects,
hence the claim holds. In the latter case, if both $U$ and $V$ are
connected, then the claim holds. Otherwise, by the same procedure,
we may decompose $U$ or $V$ into a coproduct of two non-initial objects.
This procedure stops after a finite iteration, since the number $n$ is bounded above. 
Thus, we have shown that $X$ is a coproduct of a finite number of connected objects.
\end{proof}
This proves (\ref{enum:main2}) in Theorem~\ref{th:main}. 
To show the uniqueness (\ref{enum:main3}), 
we prepare some lemmas.
\begin{lemma}\label{lem:decomp}
Let $A, B, X$, and $C$ be objects of an extensive category,
and assume $C$ connected. 
Suppose that $f:C+X \to A+B$ is an isomorphism. By connectedness of $C$,
we may assume that there is a $g:C \to A$ such that
$fi_C=i_Ag$ (by symmetry between $A$ and $B$). Then, there exists
an object $Y$ such that 
 \begin{align}
  A\cong C+Y,\ X\cong Y+B.
 \end{align}
\end{lemma}
\begin{proof}
By $fi_C=i_Ag$ we have a commutative diagram
 \begin{center}
  \begin{tikzcd}
   C\ar{r}{g}\ar{d}[swap]{\id_C} & A\ar{d}{f^{-1}i_A}\\
   C\ar{r}[swap]{i_C} & C+X.
  \end{tikzcd}
 \end{center}
Since $f^{-1}i_A$ is mono and the left vertical arrow is an isomorphism,
this diagram is a pullback.
We consider the following diagram:
 \begin{center}
  \begin{tikzcd}
   C\ar{r}{g}\ar{d}{\id_C} & A\ar{d}{f^{-1}i_A} & Y\ar{l}\ar{d}\\
   C\ar{r}[swap]{i_C} & {C+X(\overset{f}{\cong} A+B)} & X\ar{l}{i_X}\\
   Z'\ar{r}\ar{u}  & B\ar{u}{f^{-1}i_B} & Z\ar{l}\ar{u}.
  \end{tikzcd}
 \end{center}
The left top square is the observed pullback. The right top $Y$ is
defined by the pullback. By the axiom of extensive categories,
we have 
$$A \cong C+Y.$$
The bottom row is the pullback along $f^{-1}i_B$, giving $Z'$ and $Z$.
At the left bottom square, $fi_C=i_A g$ implies that $Z'$ is the 
pullback of $C\stackrel{g}{\to} A \stackrel{i_A}\to A+B$
and $B \stackrel{i_B}{\to} A+B$, and is $0$ by Lemma~\ref{lem:hom-sum}. 
Thus $B \cong Z'+Z\cong Z$ holds. The extensivity gives
$$
X\cong Y+Z \cong Y+B.
$$
\end{proof}
A connected object is cancellable.
\begin{lemma}\label{lem:cancel}
For objects $X,X'$ and a connected object $C$ in an extensive category, $C+X \cong C+X'$
implies $X \cong X'$. 
\end{lemma}
\begin{proof}
Let $f:C+X \to C+X'$ be an isomorphism. By the connectedness of $C$,
either one of the following holds.
\begin{enumerate}
 \item $fi_C=i_Cg$ holds for some $g:C \to C$.
 \item $fi_C=i_{X'}g$ holds for some $g:C \to X'$.
\end{enumerate}
In the first case, by Lemma~\ref{lem:decomp} for $A=C$ and $B=X'$, we have
$$
C\cong C+Y \mbox{ and } X \cong Y+X',
$$
where the connectedness of $C$ and
Proposition~\ref{prop:connectivity} imply that $Y\cong 0$, thus $X\cong X'$.
In the second case, by Lemma~\ref{lem:decomp} for $A=X'$ and $B=C$, we have
$$
A=X'\cong C+Y \mbox{ and } X \cong Y+B=Y+C,
$$
and hence $X'\cong X$.
\end{proof}
The following is an analogue to the unique factorization theorem,
which proves (\ref{enum:main3}) of Theorem~\ref{th:main}.
\begin{corollary}\label{cor:UFD}
In an extensive category, suppose that an object is decomposed
in two ways as 
 \begin{align}
  f:\sum_{i=1}^sC_i \overset{\sim}{\to} \sum_{i=1}^tD_i,
 \end{align}
for connected objects $C_1,\ldots,C_s$ and
$D_1,\ldots,D_t$.
Then $s=t$, and by changing the ordering, $C_i\cong D_i$ for $i=1,\ldots,s$.
\end{corollary}
\begin{proof}
Suppose that $s=0$. If $t\geq 1$, then there is a morphism $D_1 \to 0$,
which is impossible by connectivity of $D_1$ and by Proposition~\ref{prop:strict}.
Thus $t=0$, and the claim follows.
Hence, we may assume that $s,t\geq 1$.
By connectivity of $C_1$, there exists a $k$ and $g:C_1 \to D_k$
such that $fi_{C_1}=i_{D_k}g$. By Lemma~\ref{lem:decomp}, we have a $Y$
such that
$$
C_1 \cong D_k+Y.
$$
By Proposition~\ref{prop:connectivity}, $Y\cong 0$ and $C_1 \cong D_k$.
By changing the ordering, we may assume $C_1 \cong D_1$.
By Lemma~\ref{lem:cancel}, we have
 \begin{align}
  \sum_{i=2}^sC_i\cong\sum_{i=2}^tD_i.
 \end{align}
By iterating the argument, we will have $s=0$ or $t=0$.
Then, the argument at the beginning of this proof shows that
$s=t$, and $C_i\cong D_i$
for $i=1,\ldots,s$.
\end{proof}
This completes the proofs of 
(\ref{enum:main2}) and (\ref{enum:main3}) in Theorem~\ref{th:main}.

\subsection{Decomposability implies extensivity}
This section proves (\ref{enum:main1}) in Theorem~\ref{th:main}.
Throughout this section, we assume that $\cC$ is a category whose object is 
a finite coproduct of connected objects. We shall prove that $\cC$
is extensive. Every argument is about $\cC$. 
\begin{lemma}\label{lem:copro-mono} 
A coprojection
$$
X_1 \stackrel{i_{X_1}}{\to} X_1+X_2
$$
is mono.
\end{lemma}
\begin{proof}
We consider $D \rightrightarrows X_1 \to X_1+X_2$.  
To check the mono property, note that $D\cong \sum_{i}D_i$ with $D_i$
connected, and since
$$
\Hom(D, -)\cong \prod_{i}\Hom(D_i,-),
$$
it suffices to check the mono-property for each $D_i$. 
In other words, we may assume that $D$ is connected, and it suffices 
to show that if 
$f,g:D \to X_1$ satisfy $i_{X_1}f=i_{X_1}g$, then $f=g$.
By connectivity of $D$,
$$
\Hom(D,X_1) \to \Hom(D,X_1+X_2), \quad f \mapsto i_{X_1}f
$$
is injective. Thus $f=g$ follows.
\end{proof}

\begin{lemma}\label{lem:construction}
Let $A, X$ and $Y$ be objects, and
$f:A \to X+Y$ be a morphism. 
We may assume 
$$
A\cong \sum_{j\in J} A_j,
$$
with $A_j$ connected and $J$ a finite set. 
Then, by Definition~\ref{def:conn},
for each $j$, either one of the 
following two holds.
\begin{enumerate}
 \item \label{enum:X} 
 There exists $g_j:A_j \to X$ 
 such that $f\circ i_{A_j}=i_X\circ g_j$ holds.
 \item \label{enum:Y} 
 There exists $g_j:A_j \to Y$ 
 such that $f\circ i_{A_j}=i_Y \circ g_j$ holds.
\end{enumerate}
Let $J_X$ be the set of $j$ satisfying (\ref{enum:X}),
and $J_Y$ the set of $j$ satisfying (\ref{enum:Y}).
Thus, $J_X\coprod J_Y=J$.
\end{lemma}
\begin{proposition}\label{prop:AXAY}
Let $A, X$ and $Y$ be objects, and
$f:A \to X+Y$ be a morphism, as above. 
Define
$$
A_X:=\sum_{j\in J_X}A_j \mbox{ and } A_Y:=\sum_{j\in J_Y}A_j.
$$
Then, $A\cong A_X + A_Y$ holds, and the following diagram 
commutes:
 \begin{equation}\label{eq:ext-to-prove}
  \begin{tikzcd}
   A_X\ar{r}{i_{A_X}}\ar{d}[swap]{g_X} & A\ar{d}{f}& A_Y\ar{l}[swap]{i_{A_Y}}\ar{d}{g_Y}\\
   X\ar{r}[swap]{i_X} & X+Y    & Y \ar{l}{i_Y}.
  \end{tikzcd}
 \end{equation}
\end{proposition}
\begin{proof}
The claim $A\cong A_X+A_Y$ follows from $J=J_X\coprod J_Y$.
For commutativity, if $j\in J_X$, then 
\begin{equation}\label{eq:X-component}
  \begin{tikzcd}
   A_j\ar{r}{i_{A_j}}\ar{d}[swap]{g_j} & A\ar{d}{f} \\
   X\ar{r}[swap]{i_X} & X+Y    
  \end{tikzcd}
\end{equation}
commutes. Taking the coproduct over $J_X$,
we have the left commutative square of (\ref{eq:ext-to-prove}).
The commutativity at the right square follows similarly.
\end{proof}
We shall prove the first condition of Definition~\ref{def:extensive}.
\begin{proposition}\label{prop:left-pullback}
The left and right squares in (\ref{eq:ext-to-prove}) 
are pullbacks.
\end{proposition}
\begin{proof}
Take $D$, $h_1$, and $h_2$
in the following diagram with $fh_1=i_{X_1}h_2$, 
and we shall show the unique existence of $k$ that makes two triangles commute.
\begin{equation}\label{eq:D-pullback}
 \begin{tikzcd}
   D\ar[dotted]{rd}{k} \ar[bend left]{rrd}{h_1}\ar[bend right]{rdd}{h_2} & & \\ 
  & A_X\ar{r}{i_{A_X}}\ar{d}[swap]{g} & A\ar{d}{f} \\
  & X\ar{r}[swap]{i_X}  & X+Y.
 \end{tikzcd}
\end{equation}
By the same reason as in the proof of Lemma~\ref{lem:copro-mono},
we may assume that $D$ is connected. 
By connectivity of $D$ and $A=\sum_{j\in J}A_j$, 
there exists a unique $j\in J$ such that
\begin{equation}\label{eq:s}
h_1=i_{A_j}s, \ s:D \to A_j. 
\end{equation}
We claim that $j\in J_X$. 
By connectivity of $D$, we have a bijection
\begin{equation}
\Hom(D,X)\coprod \Hom(D,Y) \to \Hom(D,X+Y). 
\end{equation}
By $fh_1=i_X h_2$, $fh_1 \in \Hom(D,X+Y)$ lies in the image of 
$\Hom(D,X)$. Suppose that $j\in J_Y$. Then by 
Lemma~\ref{lem:construction}
$$
fh_1=fi_{A_j}s=i_Yg_js: D \stackrel{s}{\to} 
A_j \stackrel{g_j}{\to} Y \stackrel{i_Y}{\to} X+Y.
$$
This implies that $fh_1\in \Hom(D,X+Y)$ also comes from 
$\Hom(D,Y)$, which is absurd. Thus, $j\in J_X$.
Consequently, we have $g_j:A_j \to X$, which makes the diagram
(\ref{eq:X-component}) commute,
namely, 
$$
fi_{A_j}=i_Xg_j.
$$
We define $k:D \to A_X$ as a composition
$$
D \stackrel{s}{\to} A_j \stackrel{i_{A_j}}\longrightarrow A_X.
$$
By (\ref{eq:s}), this $k$ satisfies
$$
i_{A_X}k=i_{A_X}i_{A_j}s=i_{A_j}s=h_1, 
$$
which makes the upper triangle in (\ref{eq:D-pullback}) commute.
Since $i_{A_X}$ is mono by Lemma~\ref{lem:copro-mono} and 
Proposition~\ref{prop:AXAY}, such a $k$ is unique.
By the commutativity of the square in the diagram (\ref{eq:D-pullback}), 
$i_Xh_2=fh_1=fi_{A_X}k=i_Xgk$.
By Lemma~\ref{lem:copro-mono}, $i_X$ is mono and hence $h_2=gk$.
This shows that $A_X$ is the pullback. The same argument shows 
that $A_Y$ is the pullback.
\end{proof}
This proves the first condition of Definition~\ref{def:extensive}.
We shall prove the second condition (\ref{eq:extensive}). 
Assume that the both squares are pullbacks in (\ref{eq:extensive}).
By Proposition~\ref{prop:left-pullback}, 
we have
$A_{X_1}:=\sum_{j\in J_{X_1}} A_j$ is a pullback and hence
isomorphic to $A_1$, 
and similarly
$A_{X_2}:=\sum_{j\in J_{X_2}} A_j \cong A_2$.
Thus, 
$$
A_1\cong A_{X_1} \to A \leftarrow A_{X_2}\cong A_2
$$
is a coproduct.
Conversely, suppose that the top row of (\ref{eq:extensive})
is a coproduct. 
Let $A_1 \cong \sum_{j\in J}B_j$ with connected $B_j$,
and $A_2 \cong \sum_{k\in K}C_k$ with connected C$_k$,
with $J \cap K = \emptyset$. Then we have
$A \cong \sum_{j\in J}B_j+\sum_{k\in K}C_k$. 
The commutativity of the left square of (\ref{eq:extensive})
implies that $\sum_{j\in J}B_j$ is constructed from $A$
by the method described in Proposition~\ref{prop:AXAY},
hence $A_1$ is a pullback by Proposition~\ref{prop:left-pullback}. Similarly, $A_2$ is a pullback.
This proves that $\cC$ is extensive,
and completes the proof of (\ref{enum:main1}) in Theorem~\ref{th:main}.
\subsection{Injectivity preserved by Grothendieck groups}
This section proves Corollary~\ref{cor:main}.
We show the injectivity of (\ref{eq:grothen-inj}).
Under the condition of Theorem~\ref{th:main},
every object of $\cC$ is a finite coproduct of
connected objects in a unique way. This shows that
$(\Dec(\cC),+)$ is a free (commutative) monoid generated by the
representatives $C_i (i \in I)$ of the isomorphism classes of
connected objects, hence a (possibly infinite) direct sum
of copies of $\N$.
It is easy to see that its Grothendieck group
is the direct sum of copies of $\Z$, to which the free monoid
injects. Thus (\ref{eq:grothen-inj}) is injective.
Note that 
$\prod_{i\in I} \N$ injects to its Grothendieck group
$\prod_{i\in I} \Z$, hence
$$
(\Dec(\cC),+) \to \prod_{i\in I}(\Z,+)
$$
is injective.
The injectivity of (\ref{eq:hom-count-inj}) follows 
from the next general lemma.
\begin{lemma}
Let $M$ be a free commutative monoid
and $g:M \hookrightarrow G$ a monoid injection to a commutative group
$G$. By the universality, $g$ extends to a group homomorphism
from the Grothendieck group $(M,+,-)$ to $G$,
$$
f: (M,+,-) \to G.
$$
Then, $f$ is injective.
\end{lemma}
\begin{proof}
Let $c_j (j\in J)$ be a free generator of the free monoid $M$,
and suppose that $f$ maps $\sum_{j\in J}a_jc_j\in (M,+,-)$ to $0\in G$ with $a_j\in \Z$.
This implies that $f(\sum_{a_j>0}a_jc_j - \sum_{a_j<0}(-a_j)c_j)=0$,
hence $g(\sum_{a_j>0}a_jc_j)=g(\sum_{a_j<0}(-a_j)c_j)$,
and by the injectivity of $g$, we have   
$\sum_{a_j>0}a_jc_j=\sum_{a_j<0}(-a_j)c_j$, hence the both sides
are zero because of the definition of free generators, 
and the injectivity of $f$ follows.
\end{proof}
It is well-known that taking the Grothendieck group of a semi-ring
gives a ring, and is a functor. Thus, the rest of Corollary~\ref{cor:main} follows.
\bibliographystyle{plain}
\bibliography{sfmt-kanren-cat}

\begin{thebibliography}{1}

\bibitem{CARBONI}
A.~Carboni, S.~Lack, and R.F.C. Walters.
\newblock Introduction to extensive and distributive categories.
\newblock {\em Journal of Pure and Applied Algebra}, 84(2):145--158, 1993.

\bibitem{MONOIDALT}
M.~M. Clementino, E.~Colebunders, and W.~Tholen.
\newblock Lax algebras as spaces.
\newblock In D.~Hofmann, G.~J. Seal, and W.~Tholen, editors, {\em Monoidal
  Topology: A Categorical Approach to Order, Metric, and Topology},
  Encyclopedia of Mathematics and its Applications, chapter~V, pages 369--462.
  Cambridge University Press, 2014.

\bibitem{DAWAR}
A.~Dawar, T.~Jakl, and L.~Reggio.
\newblock {Lov\'{a}sz}-type theorems and game comonads.
\newblock In {\em Proceedings of the 36th Annual ACM/IEEE Symposium on Logic in
  Computer Science}, LICS, pages 1--13. IEEE, 2021.

\bibitem{FUJINO}
S.~Fujino and M.~Matsumoto.
\newblock {Lov\'{a}sz}'s hom-counting theorem by inclusion-exclusion principle.
\newblock arxiv:2206.01994, 2022.

\bibitem{ISBELL}
J.~Isbell.
\newblock Some inequalities in hom sets.
\newblock {\em Journal of Pure and Applied Algebra}, 76:87--110, 1991.

\bibitem{LOVASZ-OPERATION}
L.~Lov\`{a}sz.
\newblock Operations with structures.
\newblock {\em Acta Math. Acad. Sci. Hungar.}, 18:321--328, 1967.

\bibitem{PULTR}
A.~Pultr.
\newblock Isomorphism types of objects in categories determined by numbers of
  morphisms.
\newblock {\em Acta Scientiarum Mathematicarum}, 35:155--160, 1973.

\bibitem{REGGIO}
L.~Reggio.
\newblock Polyadic sets and homomorphism counting.
\newblock arxiv:2110.11061, 2021.

\end{thebibliography}


\end{document}